\include{BoxedEPS}

\documentclass{amsart}

\newtheorem{theorem}{Theorem}[section]
\newtheorem{lemma}[theorem]{Lemma}

\newtheorem{corollary}[theorem]{Corollary}

\numberwithin{equation}{section}

\begin{document}

\title [Stability Criterion for Infinite Matrices]{Stability Criterion for Convolution-Dominated
Infinite Matrices}
\author{Qiyu Sun}
\address{Department of Mathematics, University of Central Florida, Orlando,
FL 32816, USA\newline Email: \ {\tt qsun@mail.ucf.edu} }


\begin{abstract} Let $\ell^p$ be the space of all $p$-summable sequences on $\mathbb{Z}$.
An infinite matrix is said to have $\ell^p$-stability if it is bounded and has bounded inverse
on $\ell^p$. In this paper, a practical criterion is established
for the $\ell^p$-stability of convolution-dominated infinite
matrices.
  \end{abstract}

\maketitle

\section{Introduction}

Let ${\mathcal C}$ be the set of all infinite matrices
$A:=(a(j,j'))_{j,j'\in \mathbb{Z}}$ with
$$\|A\|_{\mathcal C}=\sum_{k\in \mathbb{Z}} \sup_{j-j'=k}
|a(j,j')|<\infty.$$ Let $\ell^p:=\ell^p(\mathbb{Z})$ be  the set
of all $p$-summable sequences  on $\mathbb{Z}$ with the standard
norm $\|\cdot\|_{p}$. An infinite matrix $A:=(a(j,j'))_{j,j'\in
\mathbb{Z}}\in {\mathcal C}$ defines a bounded linear operator on
$\ell^p, 1\le p\le\infty$, in the sense that
\begin{equation}
\label{matrixoperator.def} Ac= \Big(\sum_{j'\in \mathbb{Z}}
a(j,j') c(j')\Big)_{j\in \mathbb{Z}}\end{equation}
 where $c=(c(j))_{j\in
\mathbb{Z}}\in \ell^p$.
Given a summable sequence $h=(h(j))_{j\in \mathbb{Z}}\in \ell^1$,
define the {\em convolution operator} $C_h$ on $\ell^p, 1\le p\le
\infty$, by
\begin{equation}\label{convolution.def}
 C_h:\ \ell^p\ni \big(b(j)\big)_{j\in {\mathbb Z}} 
\longmapsto 
\Big(\sum_{k\in {\mathbb Z}} h(j-k) b(k)\Big)_{j\in {\mathbb
Z}}\in \ell^p.\end{equation}
 Observe that  the linear operator
associated with an infinite matrix $A\in {\mathcal C}$ is
dominated by a convolution operator in the sense that
\begin{equation}\label{dm.def}
|(Ac)(j)|\le  (C_h|c|)(j):= \sum_{j'\in \mathbb{Z}} h(j-j')
|c(j')|,\quad \ j\in \mathbb{Z}
\end{equation}
 for any sequence $c=(c(j))_{j\in \mathbb{Z}}\in \ell^p, 1\le p\le \infty$,
 where $|c|=(|c(j)|)_{j\in \mathbb{Z}}$ and the sequence  $(\sup_{j-j'=k}
  |a(j,j')|)_{k\in \mathbb{Z}}$ can be chosen to be the
 sequence $h=(h(j))_{j\in \mathbb{Z}}$ in  \eqref{dm.def}.
 So  infinite matrices in the
set ${\mathcal C}$ are said to be {\em convolution-dominated}.

\bigskip
Convolution-dominated infinite matrices were introduced by
Gohberg, Kaashoek, and Woerdeman \cite{gohberg89} as a
generalization of Toeplitz matrices. They showed that the class
${\mathcal C}$ equipped with the standard matrix multiplication
and
 the above norm $\|\cdot\|_{\mathcal C}$
 is an inverse-closed Banach subalgebra of ${\mathcal
B}(\ell^p)$ for $p=2$. Here ${\mathcal B}(\ell^p), 1\le p\le
\infty$, is the space of all bounded linear operators on $\ell^p$
with the standard operator norm, and a subalgebra ${\mathcal A}$
of a Banach algebra ${\mathcal B}$ is said to be {\em
inverse-closed} if  an operator $T\in {\mathcal A}$ has  an
inverse $T^{-1}$ in $\mathcal B$ then  $T^{-1}\in {\mathcal A}$
(\cite{connes, gelfand, naimark}). The inverse-closed property for
convolution-dominated infinite matrices was rediscovered by
Sj\"ostrand \cite{sjostrand2} with a completely different proof
and an application to a deep theorem about pseudodifferential
operators. Recently Shin and Sun \cite{shinsun09} generalized
Gohberg, Kaashoek and Woerdeman's result and proved that the class
${\mathcal C}$ is an inverse-closed Banach subalgebra of
${\mathcal B}(\ell^p)$ for any $1\le p\le \infty$. The readers may
refer to
  \cite{baskakov, fendlergrochenigappear,   kurbatov90, shinsun09, sjostrand2,  suntams}
   and  the references therein
   for related results and various generalizations on the inverse-closed property for
convolution-dominated infinite matrices.

\bigskip

Convolution-dominated infinite matrices
arise and have been used in the study of spline approximation
(\cite{deboor76,demko}), wavelets and affine frames
(\cite{chuihestockler04b, jaffard}),
 Gabor frames and non-uniform sampling (\cite{balanchl04, grochenigl03,
 grochenigl06,
 sunsiam}),  and pseudo-differential operators
 (\cite{grochenig06,
grochenigstrohmer, sjostrand, sjostrand2} and the references therein).
 Examples of convolution-dominated infinite matrices
 include the infinite matrix $\big(a(j-j')\big)_{j, j'\in
\mathbb{Z}}$ associated with convolution operators, and the
infinite matrix $\big(a(j-j') e^{-2\pi \sqrt{-1}\theta
j'(j-j')}\big)_{i,j\in\mathbb{Z}}$ associated with  twisted
convolution operators, where $\theta\in \mathbb{R}$ and the
sequence $a=(a(j))_{j\in \mathbb{Z}}$ satisfies $\sum_{j\in
\mathbb{Z}} |a(j)|<\infty$ (\cite{akramgrochenig01, grochenigl03,
jiamicchelli,
 suntams, wiener}).

\bigskip

A convolution-dominated infinite matrix $A$  is said to have {\em
$\ell^p$-stability} if there are two positive constants $C_1$ and
$C_2$ such that
\begin{equation}\label{lpstability.def}
C_1\|c\|_p\le \|Ac\|_p\le C_2\|c\|_p \quad {\rm for \ all} \ c\in
\ell^p.
\end{equation}
The $\ell^p$-stability  is one of basic assumptions for infinite
matrices arisen in the study of spline approximation,  Gabor
time-frequency analysis, nonuniform sampling, and algebra of
pseudo-differential operators, see \cite{akramgrochenig01,
balanchl04, chuihestockler04b, deboor76, demko,
fendlergrochenigappear, grochenigl03,
 grochenigl06, grochenigstrohmer,
jaffard,  jiamicchelli, shinsun09, sjostrand, sjostrand2, sunsiam,
suntams, wiener}
 and the references therein.
{\bf Practical criteria} for the $\ell^p$-stability of a
convolution-dominated infinite matrix will play important roles in
the further study of those topics.

\bigskip
However, up to the knowledge of the author, little is known about
practical criteria for the $\ell^p$-stability of an infinite
matrix. For an infinite matrix $A=(a(j-j'))_{j, j'\in \mathbb{Z}}$
associated with convolution operators, there is a very useful
criterion for its $\ell^p$-stability. It states that $A$ has
$\ell^p$-stability if and only if
 the Fourier series $\hat
a(\xi):=\sum_{j\in {\mathbb Z}} a(j) e^{-ij\xi}$ of the generating
sequence $a=(a(j))_{j\in \mathbb{Z}}\in \ell^1$ does not vanish on
the real line, i.e.,
\begin{equation}\label{hata.chara}
\hat a(\xi)\ne 0  \quad {\rm for \ all} \ \ \xi\in {\mathbb R}.
\end{equation}
 Applying this criterion for the $\ell^p$-stability,
one concludes that the spectrum $\sigma_p(C_a)$ of the convolution
operator $C_a$
 as an operator on $\ell^p$
is independent of $1\le p\le \infty$, i.e.,
\begin{equation}
\sigma_p(C_a)=\sigma_q(C_a)\quad {\rm for\ all} \ 1\le p, q\le
\infty
\end{equation}
see \cite{barnes90, hulanicki, pytlik, shinsun09} and the references
therein for the discussion on spectrum of various  convolution
operators. Applying the above criterion again, together with  the
classical Wiener's lemma (\cite{wiener}), it follows  that the inverse
of an $\ell^p$-stable convolution operator $C_a$ is  a convolution
operator $C_b$ associated with another summable sequence $b$.

\bigskip
For a  convolution-dominated infinite matrix $A=(a(j,j'))_{j,j'\in
\mathbb{Z}}$, a popular sufficient condition for its
$\ell^1$-stability and $\ell^\infty$-stability is  that $A$ is
{\em diagonal-dominated}, i.e.,
\begin{equation}\label{ddomi.def}
\inf_{j\in \mathbb{Z}}\Big( |a(j,j)|-\max \Big(\sum_{j'\ne
j}|a(j,j')|, \sum_{j'\ne j}|a(j',j)|\Big)\Big)>0.\end{equation} In
this paper, we provide a practical criterion for the
$\ell^p$-stability of convolution-dominated infinite matrices. We
 show that a convolution-dominated infinite matrix $A$ has
$\ell^p$-stability if and only if it has certain
``diagonal-blocks-dominated" property (see Theorem
\ref{criterion.tm} for  the precise statement).

\section{Main Theorem}

To state our criterion for the $\ell^p$-stability of
convolution-dominated infinite matrices, we introduce two
concepts. Given an infinite matrix $A$, define the {\em truncation
matrices} $A_s, s\ge 0$, by
$$A_s=(a(i,j)\chi_{(-s, s)}(i-j))_{i,j\in \mathbb{Z}}$$
where $\chi_E$ is the characteristic function on a set $E$. Given
$y\in \mathbb{R}$ and $1\le N\in \mathbb{Z}$, define the operator
$\chi_y^N$ on $\ell^p$ by
 $$ \chi_y^N: \ell^p\ni \big(c(j)\big)_{j\in \mathbb{Z}}\longmapsto
  \big(c(j)\chi_{(-N, N)}(j-y)\big)_{j\in \mathbb{Z}}\in \ell^p.
$$
The operator $\chi_y^N$ is a diagonal matrix ${\rm diag}
(\chi_{(-N, N)}(j-y))_{j\in \mathbb{Z}}$.

\begin{theorem} \label{criterion.tm}
Let $1\le p\le \infty$,  and $A$ be a convolution-dominated
infinite matrix in the class ${\mathcal C}$. Then the following
statements are equivalent.

\begin{itemize}
\item[{(i)}] The infinite matrix $A$ has $\ell^p$-stability.

\item[{(ii)}] There exist a positive constant $C_0$ and a positive
integer $N_0$ such that \begin{equation} \label{criterion.tm.eq1}
\|\chi_n^{2N} A\chi_n^Nc\|_p\ge C_0 \|\chi_n^N c\|_p, \quad \ c\in
\ell^p,\end{equation}
 hold for all integers $N\ge N_0$ and $n\in
N\mathbb{Z}$.

\item[{(iii)}] There exist a positive integer $N_0$ and a positive
constant $\alpha$ satisfying
\begin{equation}\label{criterion.tm.eq2}
\alpha> 2 (5+2^{1-p})^{1/p} \inf_{0\le s\le N_0}
\big(\|A-A_s\|_{\mathcal C}+ \frac{s}{N_0}\|A\|_{\mathcal C}\big)
\end{equation} such that \begin{equation}\label{criterion.tm.eq3}
\|\chi_n^{2N_0} A\chi_n^{N_0}c\|_p\ge \alpha  \|\chi_n^{N_0}
c\|_p,\quad \ c\in \ell^p,\end{equation} hold for all $n\in
N_0\mathbb{Z}$.
\end{itemize}
\end{theorem}

Taking $N_0=1$ in \eqref{criterion.tm.eq2} and
\eqref{criterion.tm.eq3}, we obtain  a sufficient condition
\eqref{wddomi.def}, which is a strong version of the
diagonal-domination condition \eqref{ddomi.def}, for  the
$\ell^\infty$-stability of a convolution-dominated infinite
matrix.

\begin{corollary}
Let  $A=(a(j,j'))_{j,j'\in \mathbb{Z}}$ be a convolution-dominated
infinite matrix in the class ${\mathcal C}$. If
\begin{equation}\label{wddomi.def}\inf_{j\in \mathbb{Z}} |a(j,j)|-2\sum_{0\ne k\in
\mathbb{Z}} \sup_{j-j'=k} |a(j, j')|>0,
\end{equation}
then $A$ has $\ell^\infty$-stability.
\end{corollary}

We say that an infinite matrix $A=(a(i,j))_{i,j\in \mathbb{Z}}$ is
a {\em band matrix} if $a(i,j)=0$ for all $i,j\in \mathbb{Z}$
satisfying $j>i-k$ or $j<i+k$.  The quantity $2k+1$ is the {\em
bandwidth} of the  matrix $A$. For a band matrix $A$ with
bandwidth $2k+1$, $A-A_s$ is the zero matrix if $s>k$. Therefore
for $N>k$,
$$\inf_{0\le
s\le N} \|A-A_s\|_{\mathcal C}+\frac{s}{N}\|A\|_{\mathcal C}\le
\frac{ k}{N} \|A\|_{\mathcal C}.$$ This, together with Theorem
\ref{criterion.tm}, gives the following sufficient condition for a
band matrix to have $\ell^p$-stability.

\begin{corollary}\label{bandmatrix.cr}
Let $1\le p\le \infty$ and  $A$ be a convolution-dominated band
matrix in the
 class ${\mathcal C}$ with bandwidth $2k+1$. If there
exists an integer $N_0>k$ such that
\begin{equation}\label{bandmatrix.cr.eq1} \|A\chi_n^{N_0}c\|_p\ge \alpha
\|\chi_n^{N_0} c\|_p,\quad \ c\in \ell^p,\end{equation} holds for
some constant $\alpha$ strictly larger than $2 (5+2^{1-p})^{1/p}
k\|A\|_{\mathcal C}/N_0$, then $A$ has $\ell^p$-stability.
\end{corollary}

If we further assume that  the infinite matrix $A$ in Corollary
\ref{bandmatrix.cr} has the form  $A=(a(j-j'))_{j,j'\in
\mathbb{Z}}$ for some finite sequence $a=(a(j))_{j\in \mathbb{Z}}$
satisfying $a(j)=0$ for $|j|>k$, then $\|A\|_{\mathcal
C}=\sum_{|j|\le k} |a(j)|$ and the condition
\eqref{bandmatrix.cr.eq1} can reformulated as follows:
\begin{equation}\label{bandmatrix.eq2}
\|\tilde A_{N_0} c\|_p\ge \frac{\gamma k}{N_0} \Big(\sum_{|j|\le
k} |a(j)|\Big) \|c\|_p, \quad c\in \mathbb{R}^{2N_0+1},
\end{equation}  holds for some $\gamma>2 (5+2^{1-p})^{1/p}$, where
\begin{equation}\label{bandmatrix.eq3}\tilde A_{N_0}=( a(j-j'))_{-N_0-k\le j\le
N_0+k,-N_0\le j'\le N_0}\end{equation} and
$$\|c\|_p=\left\{\begin{array}{ll}
(\sum_{j=-k_1}^{k_2}|c(j)|^p)^{1/p} & {\rm if}\ 1\le p<\infty\\
\sup_{-k_1\le j\le k_2} |c(j)| & {\rm if}\
p=\infty,\end{array}\right.$$
 for $c=(c(-k_1),
\cdots, c(0), \ldots, c(k_2))^T\in \mathbb{R}^{k_1+k_2+1}$. As a
conclusion from \eqref{bandmatrix.eq2} and \eqref{bandmatrix.eq3},
we see that
if $A=(a(j-j'))_{j,j'\in \mathbb{Z}}$ does not have
$\ell^p$-stability, then for any large integer $N$,
\begin{equation}\label{minimal.eq} \inf_{0\ne c\in
\mathbb{R}^{2N+1}} \frac{\|\tilde A_{N} c\|_p}{\|c\|_p} \le\frac{
 2 (5+2^{1-p})^{1/p} k}{N} \Big(\sum_{|j|\le k} |a(j)|\Big).
\end{equation} For the special case $p=2$, the above inequality
\eqref{minimal.eq} can be interpreted as that the minimal
eigenvalue of $(\tilde A_N)^T\tilde A_N$ is less than or equal to
$\frac{ \sqrt{22} k^2}{N^2} \big(\sum_{|j|\le k} |a(j)|\big)^2,$
and it can also be rewritten as
\begin{equation}\inf_{0\ne P_N\in \Pi_N}
\frac{\Big(\int_{-\pi}^\pi |\hat a(\xi)|^2 |P_N(\xi)|^2
d\xi\Big)^{1/2}} { \Big(\int_{-\pi}^\pi |P_N(\xi)|^2
d\xi\Big)^{1/2}}
 \le \frac{ \sqrt{22} k}{N}
\Big(\sum_{|j|\le k} |a(j)|\Big),
\end{equation}
where $\hat a(\xi)=\sum_{j\in\mathbb{Z}} a(j) e^{-ij\xi}$ and
$\Pi_N$ is the set of all trigonometrical polynomial of degree at
most $N$.

If the sequence $a=(a(j))_{j\in \mathbb{Z}}$ satisfies $a(0)=1,
a(-1)=-1$, and $a(j)=0$ otherwise, then the bandwidth of the
infinite matrix $A=(a(j-j'))_{j,j'\in \mathbb{Z}}$ is equal to 1,
the  norm $\|A\|_{\mathcal C}$ of the associated infinite matrix
$A$ is equal to 2,
\begin{equation}
\tilde A_N=\left (\begin{array} {cccccc} -1 & 0 & 0 & \cdots & 0 & 0\\
1 & -1 & 0 & \cdots& 0 & 0\\
0 & 1 & -1& \cdots & 0 &0\\
\vdots & \vdots & \ddots & \ddots & \vdots & \vdots\\
\vdots & \vdots & \ddots & \ddots & \ddots & \vdots\\
0 & 0 & 0 & \cdots& 1 & -1\\
0 & 0 & 0 & \cdots& 0 & 1\end{array}\right),\end{equation} and
$$\inf_{0\ne c\in \mathbb{R}^{2N+1}}
\frac{\|\tilde A_{N} c\|_p}{\|c\|_p}\ge \frac{1}{N+1},
$$
 where the last inequality holds since
the matrix
$$\tilde B_N:= \left (\begin{array} {ccccccccccc} -1 & 0 & 0 & \cdots & 0 & 0 & 0 &0 & \cdots & 0 & 0 \\
-1 & -1 & 0 & \cdots& 0 & 0 & 0 & 0 & \cdots & 0 & 0 \\
-1 & -1 & -1& \cdots & 0 &0 & 0 & 0 & \cdots & 0 & 0 \\
\vdots & \vdots & \vdots & \ddots & \vdots & \vdots & \vdots & \vdots & \ddots & \vdots & \vdots \\
-1 & -1 & -1 & \cdots& -1 & 0 & 0 & 0 & \cdots & 0 & 0 \\
0 & 0 & 0 & \cdots& 0 & 0 & 1 & 1 & \cdots & 1 & 1 \\
0 & 0 & 0 & \cdots& 0 & 0 & 0 & 1 & \cdots & 1 & 1\\
\vdots & \vdots & \vdots & \ddots & \vdots & \vdots & \vdots & \vdots & \ddots & \vdots & \vdots \\
0 & 0 & 0 & \cdots& 0 & 0 & 0 & 0 & \cdots & 1 &
1\\
0 & 0 & 0 & \cdots& 0 & 0 & 0 & 0 & \cdots & 0 &
1\\
\end{array}\right).$$ is a left inverse of the matrix $\tilde
A_N$. Therefore the order $N^{-1}$ in \eqref{minimal.eq} can not
be improved in general, but the author  believes that the  bound constant $2
(5+2^{1-p})^{1/p}$ in \eqref{criterion.tm.eq2} and
\eqref{minimal.eq} is not optimal and could be improved.

\section{Proof}

We say that  a discrete subset $\Lambda$ of $\mathbb{R}^d$  is
{\em relatively-separated}  if
\begin{equation}
R(\Lambda):=\sup_{x\in \mathbb{R}^d} \sum_{\lambda\in \Lambda}
\chi_{\lambda+[-1/2,1/2)^d}(x)<\infty
\end{equation}
(\cite{akramgrochenig01, shinsun09, suntams}).  Clearly, the set
$\mathbb{Z}$ of all integers is a relatively-separated subset of
$\mathbb{R}$ with
\begin{equation} \label{z.eq1} R(\mathbb{Z})=1.\end{equation}  Given a discrete set
$\Lambda$, let $\ell^p(\Lambda)$ be the set of all $p$-summable
sequences on the set $\Lambda$ with standard norm
$\|\cdot\|_{\ell^p(\Lambda)}$ or $\|\cdot\|_p$ for brevity.

 Given
two relatively-separated subsets $\Lambda$ and $\Lambda'$ of
$\mathbb{R}^d$,  define 
$${\mathcal C}(\Lambda, \Lambda')=\Big\{A:=\big( a(\lambda,
\lambda')\big)_{\lambda\in \Lambda, \lambda'\in \Lambda'}\Big|\
\|A\|_{{\mathcal C} (\Lambda, \Lambda')}<\infty\Big\},$$ where
\begin{eqnarray*}\|A\|_{{\mathcal C} (\Lambda, \Lambda')} & = & \sum_{k\in \mathbb{Z}^d}
\sup_{\lambda\in \Lambda, \lambda'\in \Lambda'} |a(\lambda,
\lambda')| \chi_{k+[-1/2,1/2]^d}
(\lambda-\lambda').\end{eqnarray*} It is obvious that
\begin{equation}\label{z.eq2}
{\mathcal C}(\mathbb{Z}, \mathbb{Z})= {\mathcal
C}.\end{equation}
 Given an infinite matrix
$A=(a(\lambda, \lambda'))_{\lambda\in \Lambda, \lambda'\in
\Lambda'}$,  define its {\em truncation matrices} $A_s, s\ge 0$,
by
$$ A_s=\Big(a(\lambda, \lambda') \chi_{(-s, s)^d}(\lambda-\lambda')\Big)_{\lambda\in \Lambda, \lambda'\in
\Lambda'}.$$ For any $y\in \mathbb{R}^d$ and a positive  integer
$N$, define the operator $\chi_y^N$ on $\ell^p(\Lambda)$ by
\begin{equation}
\chi_n^N: \ \ell^p(\Lambda)\ni \big(c(\lambda)\big)_{\lambda\in
\Lambda}\longmapsto \big(c(\lambda) \chi_{(-N,
N)^d}(\lambda-y)\big)_{\lambda\in \Lambda}\in \ell^p(\Lambda).
\end{equation}

In this section, we  establish the following criterion for the
$\ell^p$-stability of infinite matrices in the class ${\mathcal
C}(\Lambda, \Lambda')$,  which is a slight generalization of
Theorem \ref{criterion.tm} by \eqref{z.eq1} and \eqref{z.eq2}.

\begin{theorem}
\label{criterion.tm2} Let $1\le p\le \infty$, the subsets
$\Lambda, \Lambda'$ of $\mathbb{R}^d$ be relatively-separated, and
the infinite matrix $A$ belong to    ${\mathcal C}(\Lambda,
\Lambda')$. Then the following statements are equivalent to each
other:

\begin{itemize}
\item[{(i)}] The infinite matrix $A$ has $\ell^p$-stability, i.e.,
there exist positive constants $C_1$ and $C_2$ such that
\begin{equation}
\label{criterion.tm2.eq1} C_1\|c\|_{\ell^p(\Lambda')} \le
\|Ac\|_{\ell^p(\Lambda)} \le C_2 \|c\|_{\ell^p(\Lambda')} \quad
 {\rm for \ all} \ c\in \ell^p(\Lambda').
\end{equation}

\item[{(ii)}] There exist a positive constant $C_0$ and a positive
integer $N_0$ such that \begin{equation} \label{criterion.tm2.eq2}
\|\chi_n^{2N} A\chi_n^Nc\|_{\ell^p(\Lambda)}\ge C_0 \|\chi_n^N
c\|_{\ell^p(\Lambda')} \quad {\rm for \ all} \  \ c\in
\ell^p(\Lambda'),\end{equation} where $N_0\le N\in \mathbb{Z}$ and
$n\in N\mathbb{Z}^d$.

\item[{(iii)}] There exist a positive  integer $N_0$ and a
positive constant $\alpha$ satisfying
\begin{equation}\label{criterion.tm2.eq3}
\alpha> 2 (5+2^{1-p})^{d/p} R(\Lambda)^{1/p} R(\Lambda')^{1-1/p}
\inf_{0\le s\le N_0} \Big(\|A-A_s\|_{{\mathcal C}(\Lambda,
\Lambda')}+ \frac{d s}{N_0}\|A\|_{{\mathcal C}(\Lambda,
\Lambda')}\Big)
\end{equation} such that \begin{equation}\label{criterion.tm2.eq4}
\|\chi_n^{2N_0} A\chi_n^{N_0}c\|_{\ell^p(\Lambda)}\ge \alpha
\|\chi_n^{N_0} c\|_{\ell^p(\Lambda')}\end{equation} hold for all $
c\in \ell^p(\Lambda')$ and $n\in N_0\mathbb{Z}$.
\end{itemize}
\end{theorem}

Using the above theorem, we obtain the  following equivalence of
$\ell^p$-stability for infinite matrices having certain
off-diagonal decay, which is established  in  \cite{ald-sl,
tessera, shinsun09} for $\gamma>d(d+1), \gamma>0$,
 and $\gamma\ge
0$ respectively.

\begin{corollary}\label{stability.cr} Let $\Lambda, \Lambda'$ be
 relatively-separated subsets of $\mathbb{R}^d$, and
 $A=(a(\lambda, \lambda'))_{\lambda\in \Lambda, \lambda'\in
 \Lambda'}$  satisfy
 $$\|A\|_{{\mathcal C}_\gamma (\Lambda, \Lambda')}=
 \sum_{k\in \mathbb{Z}^d} (1+|k|)^\gamma \sup_{\lambda\in \Lambda,
 \lambda'\in \Lambda'} |a(\lambda,
 \lambda')|\chi_{k+[-1/2,1/2]^d}(\lambda-\lambda')<\infty$$
 where $\gamma>0$. Then  the $\ell^p$-stability of the infinite matrix
$A$ are equivalent to each other for different $1\le p\le \infty$.
\end{corollary}

\begin{proof}  Let $1\le p\le \infty$ and  $A$ have $\ell^p$-stability.
Then  by Theorem \ref{criterion.tm2} there exists a positive
constant $C_0$ and a positive integer $N_0$ such that
\begin{equation} \label{stability.cr.pf.eq0}
\|\chi_n^{2N} A\chi_n^Nc\|_{\ell^p(\Lambda)}\ge C_0 \|\chi_n^N
c\|_{\ell^p(\Lambda')} \quad {\rm for \ all} \  \ c\in
\ell^p(\Lambda'),\end{equation} where $N_0\le N\in \mathbb{Z}$ and
$n\in N\mathbb{Z}^d$. From the equivalence of different norms on a
finite-dimensional space, we have that
\begin{eqnarray*}
 & & ((2N)^d R(\Lambda))^{\min(1/q-1/p, 0)} \|\chi_n^N
c\|_{\ell^p(\Lambda)}   \le   \|\chi_n^N
c\|_{\ell^q(\Lambda)}\\
&  & \qquad \le  ((2N)^d R(\Lambda))^{\max(1/q-1/p, 0)} \|\chi_n^N
c\|_{\ell^p(\Lambda)} \ {\rm for\ all}\ c\in \ell^p(\Lambda),
\end{eqnarray*}
where $1\le p, q\le \infty, 1\le N\in \mathbb{Z}$ and $n\in
N\mathbb{Z}^d$ (\cite{ald-sl, shinsun09}). Therefore for $1\le
q\le \infty$,
\begin{eqnarray} \label{stability.cr.pf.eq1}
\|\chi_n^{2N} A\chi_n^Nc\|_{\ell^q(\Lambda)} & \ge &  C_0
(2N)^{-d|1/p-1/q|} R(\Lambda')^{\min(1/p-1/q,0)} \nonumber\\
& & \times
R(\Lambda)^{-\max(1/p-1/q,0)} \|\chi_n^N c\|_{\ell^q(\Lambda')} \quad {\rm for \ all}
\ \ c\in \ell^q(\Lambda'),\end{eqnarray} where $N_0\le N\in
\mathbb{Z}$ and $n\in N\mathbb{Z}^d$. We notice that
\begin{eqnarray} \label{stability.cr.pf.eq2}
\inf_{0\le s\le N} \|A-A_s\|_{{\mathcal C}(\Lambda, \Lambda')}+
\frac{ds}{N} \|A\|_{{\mathcal C}(\Lambda, \Lambda')} & \le &
\|A\|_{{\mathcal C}_\gamma (\Lambda, \Lambda')} \inf_{0\le s\le N}
s^\gamma+\frac{ds}{N}\nonumber\\
&\le &(d +1)\|A\|_{{\mathcal C}_\gamma (\Lambda, \Lambda')}
N^{-\gamma/(1+\gamma)}.
\end{eqnarray}
Thus for $1\le q\le \infty$ with
$d|1/p-1/q|<\gamma/(1+\gamma)$, it follows from
\eqref{stability.cr.pf.eq1} and \eqref{stability.cr.pf.eq2} that
there exists a sufficiently large integer $N_0$ such that
\begin{equation}\label{criterion.cr.pf.eq3} \|\chi_n^{2N}
A\chi_n^{N}c\|_{\ell^q(\Lambda)}\ge \alpha \|\chi_n^{N}
c\|_{\ell^q(\Lambda')} \end{equation} hold for all $c\in
\ell^q(\Lambda'), N\ge N_0$ and $n\in N\mathbb{Z}^d$, where
$\alpha$ is a positive constant larger than
 $2 (5+2^{1-q})^{d/q} R(\Lambda)^{1/q} R(\Lambda')^{1-1/q}
\inf_{0\le s\le N_0} \big(\|A-A_s\|_{{\mathcal C}(\Lambda,
\Lambda')}+ \frac{d s}{N_0}\|A\|_{{\mathcal C}(\Lambda,
\Lambda')}\big)$. Then by Theorem \ref{criterion.tm2}, the
infinite matrix $A$ has $\ell^q$-stability for all $1\le q\le
\infty$ with $d|1/q-1/p|<\gamma/(1+\gamma)$. Applying the above
trick repeatedly, we prove the $\ell^q$-stability of the infinite
matrix $A$ for any $1\le q\le \infty$.
\end{proof}

To prove Theorem \ref{criterion.tm2}, we  first recall some basic
properties for infinite matrices $A$ in the class ${\mathcal
C}(\Lambda, \Lambda')$ and  its truncation matrices $A_s, s\ge 0$.

\begin{lemma} \label{truncation.lem}
{\rm (\cite{shinsun09})} Let $1\le p\le \infty$, the subsets
$\Lambda, \Lambda'$ of $\mathbb{R}^d$ be relatively-separated, $A$
be an infinite matrix in  the class ${\mathcal C}(\Lambda,
\Lambda')$, and $A_s, s\ge 0$, be the truncation matrices of $A$.
 Then 
\begin{equation}\label{boundedness.eq}
\|Ac\|_{\ell^p(\Lambda)}\le  R(\Lambda)^{1/p}
R(\Lambda')^{1-1/p}\|A\|_{{\mathcal C}(\Lambda,
\Lambda')}\|c\|_{\ell^p(\Lambda')}\quad {\rm for \ all} \ c\in
\ell^p(\Lambda'),
\end{equation}
\begin{equation}\label{truncation.eq}
\lim_{s\to +\infty} \|A-A_s\|_{{\mathcal C}(\Lambda, \Lambda')}=0,
\end{equation}
\begin{equation}\label{truncation.eq4}\lim_{N\to
+\infty} \inf_{0\le s\le N}\|A-A_s\|_{{\mathcal C}(\Lambda,
\Lambda')}+\frac{ds}{N} \|A\|_{{\mathcal C}(\Lambda,
\Lambda')}=0,\end{equation}
and
\begin{equation}\label{truncation.eq3}\|A_s\|_{\mathcal C}\le
\|A\|_{\mathcal C}  \quad \ {\rm for\ all} \ s\ge 0.\end{equation}
\end{lemma}

Let $\psi_0(x_1, \ldots, x_d)=\prod_{i=1}^d \max( \min(2-2|x_i|, 1),
0)$ be a cut-off
 function on $\mathbb{R}^d$. Then
\begin{equation}\label{cutoff.eq1}
0\le \chi_{[-1/2, 1/2]^d}(x) \le \psi_0(x)\le \chi_{(-1,
1)^d}(x)\le 1 \quad {\rm for \ all} \ x\in \mathbb{R}^d,
\end{equation}
and
\begin{equation}
\label{cutoff.eq3}
 |\psi_0(x)-\psi_0(y)|\le 2 d \|x-y\|_\infty \quad \ {\rm for \ all} \ x,
y\in \mathbb{R}\end{equation}
 where $\|x\|_\infty=\max_{1\le i\le d} |x_i|$ for $x=(x_1, \ldots, x_d)$.
 Define the multiplication operator
$\Psi_n^N$ on $\ell^p(\Lambda)$ by
\begin{equation}\label{multiplication.def}
\Psi_n^N:\  \ell^p(\Lambda)\ni (c(\lambda))_{\lambda\in
\Lambda}\longmapsto
\Big(\psi_0\big(\frac{\lambda-n}{N}\big)c(\lambda)\Big)_{\lambda\in
\Lambda}\in \ell^p(\Lambda).
\end{equation}
Applying \eqref{cutoff.eq1} 
 and
\eqref{cutoff.eq3} for the cut-off function $\psi_0$, we obtain
the following properties for the multiplication operators
$\Psi_n^N, n\in N\mathbb{Z}$.

\begin{lemma} \label{multiplication.lem} 
Let $1\le N\in \mathbb{Z}$,  $\Lambda$ be a relatively-separated
subset of $\mathbb{R}^d$, and  the multiplication operators
$\Psi_n^N, n\in N\mathbb{Z}^d$, be as in
\eqref{multiplication.def}. Then
\begin{equation}\label{multiplication.lem.eq1}
\|\Psi_n^Nc\|_{\ell^p(\Lambda)}\le \|\chi_n^N
c\|_{\ell^p(\Lambda)} \quad {\rm for\  all}\  c\in \ell^p(\Lambda)
\end{equation}  where $1\le p\le \infty$,
\begin{equation} \label{multiplication.lem.eq2} \|c\|_{\ell^p(\Lambda)}\le\Big(\sum_{n\in N\mathbb{Z}^d}
\|\Psi_n^N c\|_{\ell^p(\Lambda)}^p\Big)^{1/p}\le 2^{d/p}
\|c\|_{\ell^p(\Lambda)} \quad \ {\rm for \ all} \ c\in
\ell^p(\Lambda)
\end{equation}
\begin{equation}\label{multiplication.lem.eq3}
4^{d/p} \|c\|_{\ell^p(\Lambda)}\le\Big(\sum_{n\in N\mathbb{Z}^d}
\|\Psi_n^{4N} c\|_{\ell^p(\Lambda)}^p\Big)^{1/p}\le
(5+2^{1-p})^{d/p} \|c\|_{\ell^p(\Lambda)} \ \ {\rm for \ all} \
c\in \ell^p(\Lambda),
\end{equation}
where $1\le p<\infty$,  and
\begin{equation} \label{multiplication.lem.eq4}\|c\|_{\ell^\infty(\Lambda)} = \sup_{n\in N\mathbb{Z}^d}
\|\Psi_n^N c\|_{\ell^\infty(\Lambda)}= \sup_{n\in N\mathbb{Z}^d}
\|\Psi_n^{4N} c\|_{\ell^\infty(\Lambda)} \quad \ {\rm for \ all} \
c\in \ell^\infty(\Lambda).
\end{equation}
\end{lemma}

To prove Theorem \ref{criterion.tm}, we also need the following
result.

\begin{lemma} {\rm (\cite{shinsun09})}\
Let $N\ge 1$, the subsets $\Lambda, \Lambda'$ of $\mathbb{R}^d$ be
relatively-separated,  $A$ be an infinite matrix in the class
${\mathcal C}(\Lambda, \Lambda')$, $A_N$ be the truncation matrix
of $A$, and $\Psi_n^N, n\in N\mathbb{Z}^d$, be the multiplication
operators in \eqref{multiplication.def}.  Then
\begin{equation}\label{commutator.eq}
 \|\Psi_n^NA_N -A_N\Psi_n^N\|_{{\mathcal C}(\Lambda, \Lambda')}\le  \inf_{0\le s\le N} \Big(\|A_N-A_s\|_{{\mathcal
C}(\Lambda, \Lambda')}+\frac{2ds}{N} \|A_s\|_{{\mathcal
C}(\Lambda, \Lambda')}\Big).
\end{equation}
\end{lemma}

Now we start to prove Theorem \ref{criterion.tm2}.

\begin{proof}[Proof of Theorem \ref{criterion.tm2}]\quad
(i)$\Longrightarrow$(ii):\quad By the $\ell^p$-stability of the
infinite matrix $A$, there exists a positive constant $C_0$
(independent of $n\in N\mathbb{Z}^d$ and $1\le N\in \mathbb {Z}$)
such that
\begin{equation}\label{thm1.proof.eq1}
\|A\chi_n^N c\|_{\ell^p(\Lambda)}\ge C_0 \|\chi_n^N
c\|_{\ell^p(\Lambda')} \quad {\rm for \ all} \ c\in
\ell^p(\Lambda'),
\end{equation}
 where $n\in N\mathbb{Z}^d$ and $N\ge 1$.
Noting \begin{equation}\label{thm1.proof.eq1b}\chi_n^{2N} A_N
\psi_n^N= A_N \psi_n^N\end{equation} and applying
\eqref{boundedness.eq} yield
\begin{eqnarray}\label{thm1.proof.eq2}
& & \| A \chi_n^{N} c-\chi_n^{2N} A \chi_n^Nc\|_{\ell^p(\Lambda)}\nonumber\\
 &=
&
\|(I-\chi_n^{2N}) (A-A_N)\chi_n^Nc\|_{\ell^p(\Lambda)}\nonumber\\
&\le &  R(\Lambda)^{1/p} R(\Lambda')^{1-1/p}\|A-A_N\|_{{\mathcal
C}(\Lambda,\Lambda')} \|\chi_n^{N} c\|_{\ell^p(\Lambda')},
\end{eqnarray}
where $I$ is the identity operator. Combining the estimates in
\eqref{thm1.proof.eq1} and \eqref{thm1.proof.eq2} proves that
\begin{equation} \label{thm1.proof.eq3}
\| \chi_n^{2N} A \chi_n^Nc\|_{\ell^p(\Lambda)}\ge
\big(C_0-R(\Lambda)^{1/p} R(\Lambda')^{1-1/p}\|A-A_N\|_{{\mathcal
C}(\Lambda, \Lambda')}\big)\|\chi_n^{N} c\|_{\ell^p(\Lambda')}
\end{equation}
hold for  all $c\in \ell^p(\Lambda')$, where $n\in N\mathbb{Z}^d$
and $N\ge 1$. The conclusion (ii) then follows from
\eqref{truncation.eq} and \eqref{thm1.proof.eq3}.

\bigskip
 (ii)$\Longrightarrow$(iii):\quad The implication follows from \eqref{truncation.eq4}.

\bigskip

(iii)$\Longrightarrow$(i): \quad Let $1\le p<\infty$. Take any
$n\in N_0\mathbb{Z}^d$ and $c\in \ell^p(\Lambda')$. By the
assumption (iii) for the infinite matrix $A$,
\begin{equation} \label{thm1.proof.eq5}
\|\chi_n^{2N_0} A \Psi_n^{N_0}
c\|_{\ell^p(\Lambda)}=\|\chi_n^{2N_0} A \chi_n^{N_0} \Psi_n^{N_0}
c\|_{\ell^p(\Lambda)}\ge  \alpha  \|\Psi_n^{N_0}
c\|_{\ell^p(\Lambda')}.
\end{equation}
 This together with \eqref{boundedness.eq} and \eqref{thm1.proof.eq1b} implies that
\begin{eqnarray} \label{thm1.proof.eq6}
& & \|A_{N_0}\Psi_n^{N_0} c\|_{\ell^p(\Lambda)}\nonumber \\
& = &  \|\chi_n^{2N_0} (A_{N_0}-A+A)
\Psi_n^{N_0}c\|_{\ell^p(\Lambda)}\nonumber\\
 & \ge &  \|\chi_n^{2{N_0}} A \chi_n^{N_0} \Psi_n^{N_0} c\|_{\ell^p(\Lambda)}-
\|\chi_n^{2{N_0}}
(A_{N_0}-A)\Psi_n^{N_0}c\|_{\ell^p(\Lambda)}\nonumber\\
& \ge &  \big(  \alpha- R(\Lambda)^{1/p}
R(\Lambda')^{1-1/p}\|A-A_{N_0}\|_{{\mathcal C}(\Lambda,
\Lambda')}\big)
\|\Psi_n^{N_0} c\|_{\ell^p(\Lambda')}. 
\end{eqnarray}
From \eqref{boundedness.eq} and \eqref{commutator.eq}
 it follows that
\begin{eqnarray} \label{thm1.proof.eq7}
& & \|(\Psi_n^{N_0} A_{N_0}-A_{N_0} \Psi_n^{N_0})
c\|_{\ell^p(\Lambda)}\nonumber\\
 & = & \|(\Psi_n^{N_0} A_{N_0}-A_{N_0}
\Psi_n^{N_0}) \Psi_n^{4{N_0}}c\|_{\ell^p(\Lambda)}\nonumber\\
&  \le & R(\Lambda)^{1/p} R(\Lambda')^{1-1/p}\|\Psi_n^{N_0}
A_{N_0}-A_{N_0} \Psi_n^{N_0}\|_{{\mathcal C}(\Lambda, \Lambda')}
\|\Psi_n^{4{N_0}}
c\|_{\ell^p(\Lambda')}\nonumber\\
& \le & R(\Lambda)^{1/p} R(\Lambda')^{1-1/p}\nonumber\\
& & \times
 \inf_{0\le s\le {N_0}} \Big(\|A_{N_0}-A_s\|_{\mathcal
C}+\frac{2ds}{{N_0}} \|A_{N_0}\|_{\mathcal C}\Big)
\|\Psi_n^{4{N_0}}c\|_{\ell^p(\Lambda')}. 
\end{eqnarray}

Combining 
\eqref{multiplication.lem.eq2}, \eqref{multiplication.lem.eq3},
\eqref{thm1.proof.eq6} and \eqref{thm1.proof.eq7}, we  get
\begin{eqnarray*}
& & 2^{d/p}\|A_{N_0} c\|_{\ell^p(\Lambda)}  \ge  \Big(\sum_{n\in
{N_0}\mathbb{Z}} \|\Psi_n^{N_0}
A_{N_0} c\|_{\ell^p(\Lambda)}^p\Big)^{1/p}\nonumber\\
& \ge  &  \Big(\alpha- R(\Lambda)^{1/p} R(\Lambda')^{1-1/p}
\|A-A_{N_0}\|_{{\mathcal C}(\Lambda, \Lambda')}\Big)
\Big(\sum_{n\in
{N_0}\mathbb{Z}} \|\Psi_n^{N_0}  c\|_{\ell^p(\Lambda')}^p\Big)^{1/p}\nonumber\\
& & - R(\Lambda)^{1/p} R(\Lambda')^{1-1/p}
 \inf_{0\le s\le {N_0}} \Big(\|A_{N_0}-A_s\|_{{\mathcal
C}(\Lambda, \Lambda')}+\frac{2ds}{{N_0}} \|A_{N_0}\|_{{\mathcal
C}(\Lambda, \Lambda')}\Big) \nonumber\\
& & \times \Big(\sum_{n\in
{N_0}\mathbb{Z}}  \|\Psi_n^{4{N_0}}c\|_{\ell^p(\Lambda')}^p\Big)^{1/p}\nonumber\\
& \ge & \Big(\alpha-R(\Lambda)^{1/p}
R(\Lambda')^{1-1/p}\|A-A_{N_0}\|_{{\mathcal C}(\Lambda,
\Lambda')}- (5+2^{1-p})^{1/p} R(\Lambda)^{1/p}
R(\Lambda')^{1-1/p}\nonumber\\
& & \times \inf_{0\le s\le {N_0}} \big(\|A_{N_0}-A_s\|_{{\mathcal
C}(\Lambda, \Lambda')}+\frac{2ds}{{N_0}} \|A_{N_0}\|_{{\mathcal
C}(\Lambda, \Lambda')}\big)\Big  )\|c\|_{\ell^p(\Lambda')}.
\end{eqnarray*}
Therefore
\begin{eqnarray*}
 & & \|Ac\|_{\ell^p(\Lambda)} \ge \|A_{N_0}c\|_{\ell^p(\Lambda)}- \|(A-A_{N_0})c\|_{\ell^p(\Lambda)} \nonumber\\
&\ge & 2^{-1/p}\Big (\alpha-(1+2^{d/p})R(\Lambda)^{1/p}
R(\Lambda')^{1-1/p}\|A-A_{N_0}\|_{{\mathcal C}(\Lambda, \Lambda')}\nonumber\\
& & - (5+2^{1-p})^{d/p}R(\Lambda)^{1/p}
R(\Lambda')^{1-1/p}\nonumber\\
& & \times \inf_{0\le s\le {N_0}} \big(\|A_{N_0}-A_s\|_{{\mathcal
C}(\Lambda, \Lambda')}+\frac{2ds}{{N_0}}
\|A_{N_0}\|_{{\mathcal C}(\Lambda, \Lambda')}\big) \Big ) \|c\|_{\ell^p(\Lambda')}\nonumber\\
& \ge & 2^{-d/p} \Big(\alpha- 2
(5+2^{1-p})^{1/p}R(\Lambda)^{1/p}\nonumber\\
& &\times R(\Lambda')^{1-1/p} \inf_{0\le s\le {N_0}}
\big(\|A-A_s\|_{{\mathcal C}(\Lambda, \Lambda')}+\frac{ds}{{N_0}}
\|A\|_{{\mathcal C}(\Lambda, \Lambda')}\big)\Big  )
\|c\|_{\ell^p(\Lambda')},
\end{eqnarray*}
and the conclusion (i) for $1\le p<\infty$ follows.

The conclusion (i) for $p=\infty$ can be proved  by similar
argument. We omit the details here.
\end{proof}

The author thanks Professors Deguang Han,  Zuhair M. Nashed,
Xianliang Shi, and Wai-Shing Tang for their discussion and
suggestions in preparing the manuscript.

\end{document}